\documentclass[10pt,a4paper]{amsart}

\usepackage{graphicx,color}
\usepackage{amsmath,amsthm,amsfonts,
	amssymb,
	mathrsfs,array,
	comment, 
	threeparttable
}
\usepackage{pdflscape}
\usepackage{multirow}
\usepackage{booktabs}
\usepackage{hyperref}

\theoremstyle{definition}
\newtheorem{df}{Definition}[section]

\theoremstyle{plain}
\newtheorem{thm}[df]{Theorem}
\newtheorem{prop}[df]{Proposition}

\newtheorem{lem}[df]{Lemma}
\newtheorem{cor}[df]{Corollary}
\newtheorem{fact}[df]{Fact}
\newtheorem*{prob}{Problem}

\theoremstyle{remark}
\newtheorem{rem}[df]{Remark}

\makeatletter

\@addtoreset{equation}{section}
\makeatother

\newcommand{\N}{\mathbb{N}}
\newcommand{\Z}{\mathbb{Z}}
\newcommand{\R}{\mathbb{R}}
\newcommand{\C}{\mathbb{C}}
\newcommand{\Ha}{\mathbb{H}}

\newcommand{\im}{\sqrt{-1}}

\newcommand{\Ad}{\mathop{\mathrm{Ad}}\nolimits}
\newcommand{\Ann}{\mathop{\mathrm{Ann}}\nolimits}

\newcommand{\diag}{\mathop{\mathrm{diag}}\nolimits}

\newcommand{\gr}{\mathop{\mathrm{gr}}\nolimits}
\newcommand{\Hom}{\mathop{\mathrm{Hom}}\nolimits}

\newcommand{\rank}{\mathop{\mathrm{rank}}\nolimits}

\newcommand{\Span}{\mathop{\mathrm{Span}}\nolimits}

\newcommand{\SU}{\mathrm{SU}}
\newcommand{\SO}{\mathrm{SO}}
\newcommand{\Spin}{\mathrm{Spin}}
\newcommand{\Sp}{\mathrm{Sp}}
\newcommand{\Mp}{\mathrm{Mp}}
\newcommand{\E}{\mathrm{E}}

\newcommand{\e}{\mathfrak{e}}
\newcommand{\f}{\mathfrak{f}}
\newcommand{\g}{\mathfrak{g}}

\allowdisplaybreaks[1]

\newcommand{\Addresses}{{
  \bigskip
  \footnotesize
  \textsc{Graduate School of Mathematical Sciences, the University of Tokyo, 3-8-1 Komaba, Meguro, Tokyo 153-8914, Japan}\par\nopagebreak
  \textit{E-mail}: \texttt{tamori@ms.u-tokyo.ac.jp}
}}

\title[Classification of minimal representations]%
	{Classification of minimal representations of real simple Lie groups
	}
\author[H. Tamori]{Hiroyoshi TAMORI}
\date{}
\newcommand{\keyword}[1]{\textit{Key Words and Phrases:} #1}
\newcommand{\MSC}[1]{\textit{Mathematics Subject Classification $2010$:} #1}
\begin{document}
\begin{abstract}
Based on an idea in [Gan--Savin, Represent. Theory (2005)], 
we give a classification of minimal representations of connected simple real Lie groups not of type $A$. 
Actually, we prove that there exist no new minimal representations up to infinitesimal equivalence. \\
\MSC{22E46.}\\
\keyword{Minimal representation; Reductive group.}
\end{abstract}
\maketitle
\section{introduction}\label{section:intro}
Let $G$ be a connected simple real Lie group not of type $A$, and 
$\mathfrak{g}=\mathfrak{k}+\mathfrak{p}$ a Cartan decomposition of $\mathfrak{g}:=Lie(G)$, and $K$ the analytic subgroup with Lie algebra $\mathfrak{k}$. 
For simplicity, assume that the complexification $\mathfrak{g}_{\C}$ of $\mathfrak{g}$ is simple. 
An irreducible admissible representation of $G$ is called minimal 
if the annihilator of the underlying $(\mathfrak{g}_{\C},K)$-module is the Joseph ideal $J_0$ \cite{Jos76}, 
which is the unique completely prime two-sided ideal whose associated variety is the closure of the minimal nilpotent orbit $\mathcal{O}_{\min}$ in $\mathfrak{g}_{\C}$ \cite[Theorem 3.1]{GS04} (see 
Section \ref{section:preliminaries} for the precise definitions). 

For $G=\Mp(n,\R)$ (the connected double cover of the real symplectic group $\Sp(n,\R)$), 
the two irreducible components of 
the Weil representation, which is also referred to as the harmonic, Segal--Shale--Weil, oscillator, or metaplectic representation, 
are classical examples of minimal representations. 

Unitary minimal representations are the smallest ``unipotent'' representations, 
and in the viewpoint of the Kirillov--Kostant orbit method they are supposed to be attached to real orbits included in $\mathcal{O}_{\min}$. 
From such qualities, 
minimal representations for various simple real Lie groups have been studied and constructed in various ways (see \cite{BSZ06, BZ91, BJ98, Bry98, BK94D, BK94E, BK94M, BK95, EPWW85, GS05, Gon82, GW94, GW96, HKM14, HKMO12, Hua95, HL99, Kaz90, KS90, Kob11a, KM11, KO98, KO03i, KO03ii, KO03iii, Kos90, Li00, LS08, Sab96, Sal06, Tor97, Vog81, Vog94} for example). 

An easy necessary condition for the existence of minimal $(\mathfrak{g}_{\C}, K)$-modules comes from their associated varieties. 
That is, if the intersection of the minimal nilpotent orbit $\mathcal{O}_{\min}$ and $\mathfrak{p}_{\C}$, as a subset of $\mathfrak{g_{\C}}$, is empty, 
then there exist no minimal $(\mathfrak{g}_{\C},K)$-modules. 
Moreover, R.~Howe and D.~Vogan \cite[Theorem 2.13]{Vog81} proved that for $\mathfrak{g}=\mathfrak{so}(p,q)$ with $p, q\ge 4$ and $p+q$ odd, there exists no irreducible $(\mathfrak{g}_{\C},K)$-module whose Gelfand--Kirillov dimension is $p+q-3$, namely half the complex dimension of $\mathcal{O}_{\min}$. 

The above previous studies on the construction of minimal representations imply that the converse to the non-existence statement holds: 
if $\mathfrak{g}$ is not of type $A$, 
not isomorphic to $\mathfrak{so}(p,q)$ with $p, q\ge 4$ and $p+q$ odd, and 
$\mathcal{O}_{\min}\cap\mathfrak{p}_{\C}\neq\emptyset$, 
then there exists a minimal $(\mathfrak{g}_{\C}, K)$-module for some covering $K$. 

The aim of this article is to give a negative answer, which we could not find in the literature, to the following problem. 
\begin{prob}\label{prob}
Are there new minimal $(\mathfrak{g}_{\C},K)$-modules up to isomorphism?
\end{prob}
The classification is given in Theorem \ref{main-thm}. 
The number of the isomorphism classes of minimal $(\mathfrak{g}_{\C},K)$-modules for simply connected $K$ is given in Table \ref{number}. 
	
	\begin{table}[ht]
	\begin{center}
	\label{number}
	\caption{The number of the isomorphism classes of minimal $(\mathfrak{g}_{\C},K)$-modules for simply connected $K$}
	\begin{tabular}{cc}\toprule
		$\mathfrak{g}$&number
		\\\hline\hline
		$\mathfrak{sp}(n,\R)(n\ge 2)$&$4$
		\\\hline
		$\mathfrak{so}(p,2)(p\ge 5), \mathfrak{so}^*(2n)(n\ge 4), \e_{6(-14)}, \e_{7(-25)}$&$2$
		\\\hline 
		$\mathfrak{so}(p,q)(p, q\ge 3, p+q\ge 8 \text{ even}),$&$1$\\
		$\mathfrak{so}(p,3)(p\ge 4 \text{ even}), \e_{6(6)}, \e_{6(2)},$\\
		$\e_{7(7)}, \e_{7(-5)}, \e_{8(8)}, \e_{8(-24)}, \f_{4(4)}, \g_{2(2)}$
		\\\hline
		$\mathfrak{sp}(n)(n\ge 2), \mathfrak{so}(n)(n\ge 7), \e_6, \e_7, \e_8, \f_4, \g_2,$&$0$\\
		$\mathfrak{so}(n,1) (n\ge 6), \mathfrak{sp}(p,q) (p,q\ge 1), \e_{6(-26)}, \f_{4(-20)},$\\
		$\mathfrak{so}(p,q) (p,q\ge 4, p+q \text{ odd})$\\\hline
		$\mathfrak{sp}(n,\C)(n\ge 2)$&$2$
		\\\hline
		$\mathfrak{so}(n,\C)(n\ge 7), \e_6(\C), \e_7(\C), \e_8(\C), \f_4(\C), \g_2(\C)$&$1$\\\bottomrule
	\end{tabular}
	\end{center}
	\end{table}

By the classification and previous work on each minimal representation, 
it follows that 
the $K$-type decomposition of any minimal $(\mathfrak{g}_{\C},K)$-module is written as 
\begin{equation}\notag
\bigoplus_{n\ge 0}V(\mu_0+n\beta)
\end{equation}
where $\beta$ is a highest weight of $K$-module $\mathfrak{p}_{\C}$ and 
$V(\mu_0+n\beta)$ denotes the irreducible $K$-module with highest weight $\mu_0+n\beta$ (see Theorem \ref{main-thm}(\ref{K-type})). 
The minimal $K$-type $V(\mu_0)$ is written in Table \ref{classification1} or \ref{classification2}. 
Moreover, 
all minimal $(\mathfrak{g}_{\C},K)$-modules are unitarizable (see Corollary \ref{unitarizability}) 
although unitarizability is not required in the definition of minimal $(\mathfrak{g}_{\C},K)$-modules (see Definition \ref{df:min}), 
and the unitary globalizations are still irreducible as representations of minimal parabolic subgroups when $\mathfrak{g}_{\C}$ is simple and the real rank $\R$-$\rank(\mathfrak{g})$ is not less than three (see Corollary \ref{irreducibility}). 

Let us mention relevant classification results. 
When the Lie group $G$ is the universal cover of $\Sp(n,\R)$ or that of $\SO_0(p,q)$, J.-S.~Huang and J.-S.~Li \cite{HL99} classified irreducible unitary representations of which the associated varieties of the annihilators are $\mathcal{O}_{\min}$ in terms of local theta lifting. 
Since the underlying $(\mathfrak{g}_{\C}, K)$-modules of them agree with all minimal $(\mathfrak{g}_{\C}, K)$-modules, 
an irreducible $(\mathfrak{g}_{\C}, K)$-module whose associated variety is $\mathcal{O}_{\min}$ 
is unitarizable exactly when the annihilator is completely prime (see Remark \ref{unitarizability-complete prime} for the precise statement). 
We must note that 
for $\mathfrak{g}=\e_{6(2)}, \e_{7(-5)}$ or $\e_{8(-24)}$ minimal $(\mathfrak{g}_{\C},K)$-modules were classified by W.~T.~Gan and G.~Savin \cite[Proposition 12.11]{GS05}, 
and that for exceptional groups with finite center and $\R$-$\rank(\mathfrak{g})\ge 3$, 
minimal unitary representations were classified by H.~Salmasian \cite[Proposition 5]{Sal06}. 

Our argument is based on the idea of W.~T.~Gan and G.~Savin. 
They obtained a relation of Casimir elements for two simple components of $\mathfrak{k}_{\C}$ modulo the Joseph ideal $J_0$, and considered $K$-types satisfying the relation. 
The key point of their proof is to apply a proposition by B.~Kostant: if two minimal $(\mathfrak{g}_{\C},K)$-modules have a common $K$-type, then they are isomorphic when $\mathfrak{g}_{\C}$ is simple not of type $A$, $\mathcal{O}_{\min}\cap\mathfrak{p}_{\C}\neq\emptyset$ and $G/K$ is a non-Hermitian symmetric space. 
As is implicit in \cite{GS05}, the proposition can be generalized to any connected simple real Lie group $G$ not of type $A$ similarly. 
For the convenience of the reader, we write the proof in Section \ref{section:criterion} (see Proposition \ref{criterion}). 

In Section \ref{section:classification}, the classification (Theorem \ref{main-thm}) is given by applying Proposition \ref{criterion}. 
We will consider the action of the center $Z(\mathfrak{k}_{\C})$ of the universal enveloping algebra of $\mathfrak{k}_{\C}$, 
then show that for each minimal $(\mathfrak{g}_{\C},K)$-module $\pi$, there exists a $K$-type in $\pi$ which also occurs in one already known (see Theorem \ref{same-K}). 

{\it Notation}: $\N:=\{0,1,2,\ldots\}$.

\section*{Acknowledgements}
The author is deeply grateful to his advisor Prof. Toshiyuki Kobayashi for helpful comments and warm encouragement. 
The author expresses his sincere thanks to Dr. Yoshiki Oshima and Dr. Masatoshi Kitagawa for inspiring discussions. 
This work was supported by JSPS KAKENHI Grant Number 17J01075 and the Program for Leading Graduate Schools, MEXT, Japan. 

\section{preliminaries}\label{section:preliminaries}
In this section, we recall the definition of minimal representations.

Let $\mathfrak{g}_{\C}$ be a simple complex Lie algebra and $U(\mathfrak{g}_{\C})$ the universal enveloping algebra of $\mathfrak{g}_{\C}$. 
We identify $\mathfrak{g}_{\C}$ and its dual $\mathfrak{g}_{\C}^{\vee}$ via the Killing form. 
A.~Joseph \cite{Jos76} proved that if $\mathfrak{g}_{\C}$ is not of type $A$
there exists uniquely a completely prime two-sided ideal $J_0$ 
of which the associated variety is the closure of the minimal nilpotent orbit $\mathcal{O}_{\min}$ in $\mathfrak{g}_{\C}$ (cf. \cite[Theorem 3.1]{GS04}). 
Moreover, the ideal $J_0$ was shown to be maximal (hence primitive) and 
its infinitesimal character was computed (see Table \ref{infinitesimal character}, which is cited from \cite[Table]{Jos76}). 
The ideal $J_0$ is called the Joseph ideal. 

Let $G$ be a connected simple real Lie group, and $\mathfrak{g}=Lie(G)$. 
We fix a Cartan involution $\theta$ of $\mathfrak{g}$. 
We write $\mathfrak{g}=\mathfrak{k}+ \mathfrak{p}$ for the Cartan decomposition corresponding to the Cartan involution $\theta$. 
Let us denote by $\mathfrak{g}_{\C}=\mathfrak{k}_{\C}+ \mathfrak{p}_{\C}$ its complexification and by $K$ the analytic subgroup of $G$ with $Lie(K)=\mathfrak{k}$, which is not compact if and only if $G$ is simply connected and $G/K$ is Hermitian. 

We will consider $(\mathfrak{g}_{\C},K)$-modules for not necessarily compact $K$:
\begin{df}\label{df:g,K}
Let $V$ be a $\C$-vector space which is a $\mathfrak{g}_{\C}$-module and a $K$-module. 
We will write $\Ad$ 
for the adjoint representation of $K$ on $\mathfrak{g}_{\C}$.
Then $V$ is called $(\mathfrak{g}_{\C},K)$-module if the following hold: 
\begin{itemize}
\item $k(Xv)=(\Ad(k)X)(kv)$ for all $k\in K, X\in \mathfrak{g}_{\C}$ and $v\in V$.
\item The linear span $\Span(Kv)$ of $Kv$ is finite-dimensional and 
when we put on $\Span(Kv)$ the topology of Hausdorff topological vector space 
the action of $K$ on the submodule $\Span(Kv)$ is continuous, 
and $\Span(Kv)$ decomposes to the direct sum of irreducible $K$-submodules, for all $v\in V$.
\item $\left.\frac{d}{dt}\right|_{t=0}(\exp(tX)v)=Xv$ for all $X\in \mathfrak{k}, v\in V$.
\end{itemize}
\end{df}

\begin{df}\label{df:min}
Assume that $\mathfrak{g}$ is a simple real Lie algebra not of type $A$. 
\begin{enumerate}
\item 
Let $\mathfrak{g}_{\C}$ be simple. 
An irreducible $(\mathfrak{g_{\C}},K)$-module is called \emph{minimal} if its annihilator is the Joseph ideal $J_0$. 
\item 
Let $\mathfrak{g}_{\C}$ be not simple. In other words, $\mathfrak{g}$ carries a complex Lie algebra structure. 
Then we can take an isomorphism from $\mathfrak{g}_{\C}$ to $\mathfrak{g}\oplus \mathfrak{g}$ 
as complex Lie algebras and 
it induces the isomorphism $U(\mathfrak{g}_{\C})\cong U(\mathfrak{g})\otimes U(\mathfrak{g})$. 
An irreducible $(\mathfrak{g_{\C}},K)$-module is called \emph{minimal} if its annihilator is equal to $J_0\otimes U(\mathfrak{g})+ U(\mathfrak{g})\otimes J_0$ via the isomorphism. 
This condition is independent of the choice of isomorphism $\mathfrak{g}_{\C}\cong \mathfrak{g}\oplus \mathfrak{g}$ because of the uniqueness of the Joseph ideal. 
\end{enumerate}
\end{df}
\begin{rem}\label{bimodule}
Let $\mathfrak{g}$ be a complex simple Lie algebra not of type $A$. 
\begin{enumerate}
\item
The two-sided ideal $J_0\otimes U(\mathfrak{g})+ U(\mathfrak{g})\otimes J_0$ is the unique completely prime one whose associated variety is the closure of $\mathcal{O}_{\min}\times\mathcal{O}_{\min}$. 
\item\label{bimodule2}
Via an isomorphism $\mathfrak{g}_{\C}\cong \mathfrak{g}\oplus \mathfrak{g}$ which gives the one between $\mathfrak{k}_{\C}$ and the diagonal subalgebra $\diag \mathfrak{g}$, 
the Harish-Chandra bimodule $U(\mathfrak{g})/J_0$ is a minimal $(\mathfrak{g}_{\C},K)$-module which has a nonzero $K$-fixed vector.
\end{enumerate} 
\end{rem}

	\begin{table}[h]
	\begin{center}
	\begin{threeparttable}
	\caption{The infinitesimal characters of minimal $(\mathfrak{g}_{\C},K)$-modules}
	\label{infinitesimal character}
	\begin{tabular}{cc}\toprule
		$\mathfrak{g}_{\C}$& representative of infinitesimal character
		\\\hline\hline
		$\mathfrak{so}(2n+1,\C)(n\ge 3)$&$\sum_{i=1}^{n-3}\omega_i+1/2\ \omega_{n-2}+1/2\ \omega_{n-1}+\omega_n$\\\hline
		$\mathfrak{sp}(n,\C)(n\ge 2)$&$\sum_{i=1}^{n-1}\omega_i+1/2\ \omega_n$\\\hline
		$\mathfrak{so}(2n,\C)(n\ge 4)$&$\sum_{i=1}^{n-3}\omega_i+\omega_{n-1}+\omega_{n}$\\\hline
		$\e_n(\C)(n=6,7,8)$&$(\sum_{i=1}^{3}+\sum_{i=5}^{n})\omega_i$\\\hline
		$\f_4(\C)$&$1/2\ \omega_1+1/2\ \omega_2+\omega_3+\omega_4$\\\hline
		$\g_2(\C)$&$\omega_1+1/3\ \omega_2$\\\bottomrule
	\end{tabular}
	\small
	\begin{tablenotes}
	\item[*]
	We follow the notation of Bourbaki \cite[pp.250--274]{Bou68} for fundamental weights $\omega_i$.
	\item[*]
	When $\mathfrak{g}_{\C}$ is not simple, a representative of the infinitesimal character is given by the sum of those of simple components. For example, when $\mathfrak{g}_{\C}\cong\g_2(\C)\oplus \g_2(\C)$ the infinitesimal character is the orbit through $(\omega_1+1/3\ \omega_2, \omega_1+1/3\ \omega_2)$.
	\end{tablenotes}
	\end{threeparttable}
	\end{center}
	\end{table}

\begin{df}
An irreducible representation $\Pi$ of $G$ on a complete locally convex Hausdorff topological vector space over $\C$ is called \emph{admissible} if the following hold:
\begin{itemize}
\item For any irreducible finite-dimensional representation $\tau$ of $K$, the space of $K$-intertwining operators from $\tau$ to $\Pi$ is finite-dimensional. 
\item The space of $K$-finite vectors is dense. 
\item For any $K$-finite vector $v$, the linear span $\Span(Kv)$ of $Kv$ decomposes to the direct sum of irreducible $K$-submodules. 
\end{itemize}
When $G$ is a connected simple real Lie group not of type $A$, 
an irreducible admissible representation of $G$ is called \emph{minimal} if the underlying $(\mathfrak{g}_{\C}, K)$-module is minimal. 
\end{df}
\section{criterion for isomorphism}\label{section:criterion}
Let $G$ be a connected simple real Lie group not of type $A$, and $\mathfrak{g}, \mathfrak{k}, \mathfrak{p}, K, \mathcal{O}_{\min}, J_0$ as in Section $\ref{section:preliminaries}$. 
In this section, we give a proof for the $K$-multiplicity freeness of minimal $(\mathfrak{g}_{\C},K)$-modules and a criterion for isomorphism of two minimal $(\mathfrak{g}_{\C},K)$-modules (see Proposition \ref{criterion}). 
They were proven by B.~Kostant and the proof is written in a paper by W.~T.~Gan and G.~Savin \cite{GS05} 
when $\mathfrak{g}_{\C}$ is simple, $\mathcal{O}_{\min}\cap\mathfrak{p}_{\C}\neq\emptyset$ and $G/K$ is non-Hermitian. 
The proof for the general case is similar, but we write it for the convenience of the reader. 
\begin{prop}[{\cite[Proposition 4.10]{GS05}}]\label{criterion}
Let $G$ be a connected simple real Lie group not of type $A$. 
Any minimal $(\mathfrak{g}_{\C},K)$-module is $K$-multiplicity free: every irreducible representation of $K$ occurs with multiplicity at most one. 
Moreover, if two minimal $(\mathfrak{g}_{\C},K)$-modules have a common $K$-type, 
then they are isomorphic. 
\end{prop}
For the proof of Proposition \ref{criterion}, we use Lemma \ref{K-inv} and Facts \ref{non-exist} and \ref{LM} below. 

Let us fix a maximally noncompact Cartan subalgebra $\mathfrak{h}=\mathfrak{t}'\oplus\mathfrak{a}$ with $\mathfrak{t}'\subset\mathfrak{k}, \mathfrak{a}\subset\mathfrak{p}$. 
Choose a positive system for $(\mathfrak{g}_{\C}, \mathfrak{h}_{\C})$ which is compatible with some positive system of the restricted root system for $(\mathfrak{g}, \mathfrak{a})$. 
Write $\psi$ for the highest root of $\mathfrak{g}_{\C}$ (resp. $\mathfrak{g}$) when $\mathfrak{g}_{\C}$ is (resp. is not) simple.
Then the following fact gives characterizations of the condition $\mathcal{O}_{\min}\cap\mathfrak{p}_{\C}=\emptyset$.

\begin{fact}[{\cite[Corollary 5.9]{KO15} or \cite[Proposition 4.1]{Oku15}}]\label{non-exist}
Let $\mathfrak{g}$ be a simple real Lie algebra. Suppose that $\mathfrak{g}_{\C}$ is simple. 
The following are equivalent: 
\begin{enumerate}
\item[(i)] $\mathcal{O}_{\min}\cap\mathfrak{p}_{\C}=\emptyset$.
\item[(ii)] The highest root $\psi$ for $(\mathfrak{g}_{\C},\mathfrak{h}_{\C})$ is not a real root $(\text{i.e. }\psi|_{\mathfrak{t}'}\neq 0)$.
\item[(iii)]
	$\mathfrak{g}$ is compact or isomorphic to $\mathfrak{sl}(n,\Ha) (n\ge 2), \mathfrak{so}(n,1) (n\ge 4), \mathfrak{sp}(p,q) (p,q\ge 1), \e_{6(-26)},$ or $\f_{4(-20)}$.
\end{enumerate}
\end{fact}
\begin{rem}
Fact \ref{non-exist} also follows from 
\cite[Proposition 4.3 and Section 10.2]{Tor97} 
by using the Kostant--Sekiguchi correspondence. 
\end{rem}

We next consider the structure of the algebra $(U(\mathfrak{g}_{\C})/J_0)^K$ of $K$-invariants.
\begin{lem}[{\cite[Lemma 4.11]{GS05}}]\label{K-inv}
Let $\Omega_{\mathfrak{k}_{\C}}$ be the Casimir element in $U(\mathfrak{k}_{\C})$. 
\begin{enumerate}
\item\label{K-inv1}
Suppose that $\mathfrak{g}_{\C}$ is simple. 
\begin{enumerate}
\item If $\mathcal{O}_{\min}\cap\mathfrak{p}_{\C}=\emptyset$, then $(U(\mathfrak{g}_{\C})/J_0)^K$ is equal to $\C$. 
\item
If $\mathcal{O}_{\min}\cap\mathfrak{p}_{\C}\neq\emptyset$, then $(U(\mathfrak{g}_{\C})/J_0)^K$ is the polynomial algebra in
\begin{align*}
T:=\begin{cases}
Z 
& \text{if $G/K$ is Hermitian}\\
\Omega_{\mathfrak{k}_{\C}}
& \text{if $G/K$ is non-Hermitian.}
\end{cases}
\end{align*}
Here $Z$ denotes any nonzero element in the center of $\mathfrak{k}_{\C}$.
\end{enumerate}
\item\label{K-inv2}
Suppose that $\mathfrak{g}_{\C}$ is not simple. Let $\diag\mathfrak{g}$ be the diagonal subalgebra of $\mathfrak{g}\oplus\mathfrak{g}$. 
Then the algebra $A:=\left(U(\mathfrak{g})\otimes U(\mathfrak{g})/(J_0\otimes U(\mathfrak{g})+U(\mathfrak{g})\otimes J_0)\right)^{\diag\mathfrak{g}}$ of $\diag\mathfrak{g}$-invariants is the polynomial algebra in the Casimir element $\Omega_{\diag\mathfrak{g}}$ for $\diag\mathfrak{g}$. 
\end{enumerate}
\end{lem}
\begin{proof}
We first assume that $\mathfrak{g}_{\C}$ is simple. 
Let $U_n(\mathfrak{g}_{\C})$ denote the subspace of $U(\mathfrak{g}_{\C})$ spanned by the products of $n$ or less elements of $\mathfrak{g}_{\C}$. 
We will write $S^n(\mathfrak{g}_{\C})$ for the $n$-th symmetric power of $\mathfrak{g}_{\C}$, and $S(\mathfrak{g}_{\C})$ for the symmetric algebra. 
Let $\gr_n$ denote the natural isomorphism from $U_n(\mathfrak{g}_{\C})/U_{n-1}(\mathfrak{g}_{\C})$ to $S^n(\mathfrak{g}_{\C})$ as $\mathfrak{g}_{\C}$-modules. 
We see $\gr_{m+n}(XY)=\gr_m(X)\gr_n(Y)$ for all $X\in U_m(\mathfrak{g}_{\C}), Y\in U_n(\mathfrak{g}_{\C})$. 
Since $\gr_n$ induces the isomorphism 
\begin{align*}
U_n(\mathfrak{g}_{\C})/(U_{n-1}(\mathfrak{g}_{\C})+(J_0\cap U_n(\mathfrak{g}_{\C})))\cong S^n(\mathfrak{g}_{\C})/(\gr J_0\cap S^n(\mathfrak{g}_{\C})),
\end{align*}
it suffices to prove 
\begin{align}\label{sym-version}
(S(\mathfrak{g}_{\C})/\gr J_0)^K=
\begin{cases}
\C&\text{if $\mathcal{O}_{\min}\cap\mathfrak{p}_{\C}=\emptyset$}\\
\C[\gr_1Z]&\text{if $\mathcal{O}_{\min}\cap\mathfrak{p}_{\C}\neq\emptyset$ and $G/K$ is Hermitian}\\
\C[\gr_2\Omega_{\mathfrak{k}_{\C}}]&\text{otherwise.}
\end{cases}
\end{align}

D.~Garfinkle \cite[Chapter I\hspace{-.1em}I\hspace{-.1em}I and V]{Gar82} proved 
that the associated graded ideal $\gr J_0$ of the Joseph ideal is the ideal defined by the minimal nilpotent orbit $\mathcal{O}_{\min}$, which is prime, 
and that the $\mathfrak{g}_{\C}$-module $S^n(\mathfrak{g}_{\C})$ is the direct sum of $\gr J_0\cap S^n(\mathfrak{g}_{\C})$ and the irreducible $\mathfrak{g}_{\C}$-module $W(n\psi)$ with highest weight $n\psi$. 
In particular, $S(\mathfrak{g}_{\C})/\gr J_0 \cong \bigoplus_{n\ge 0}W(n\psi)$ as $\mathfrak{g}_{\C}$-modules. 

Since we consider only the action of the adjoint group, we may assume that $G$ has a complexification. 
The Cartan--Helgason theorem (cf. \cite[Theorem 8.49]{Kna02}) 
asserts that the space of $K$-fixed vectors $W(n\psi)^K$ is at most one-dimensional, 
and that $W(n\psi)^K$ is nonzero 
if and only if $\psi$ is a real root and $(n\psi|_{\mathfrak{a}}, \alpha)/(\alpha, \alpha)\in \Z$ for every restricted root $\alpha$, where $(\cdot, \cdot)$ denotes the inner product on $\mathfrak{a}^{\vee}$ given by the Killing form of $\mathfrak{g}$. 
Hence there exists a unique nonnegative integer $n_0$ such that $W(n\psi)^K$ is nonzero 
if and only if $n\in n_0\N$. 

When $\mathcal{O}_{\min}\cap\mathfrak{p}_{\C}$ is empty, Fact \ref{non-exist} asserts that $\psi$ is not a real root. Therefore $n_0$ is zero and $(S(\mathfrak{g}_{\C})/\gr J_0)^K=\C$. 

From now, we will assume $\mathcal{O}_{\min}\cap\mathfrak{p}_{\C}\neq \emptyset$. 
Our next claim is $\gr_2\Omega_{\mathfrak{k}_{\C}}\not\in \gr J_0$. 
Conversely, suppose $\gr_2\Omega_{\mathfrak{k}_{\C}}\in \gr J_0$. 
Since $\psi$ is a real root by Fact \ref{non-exist}, we can take a highest root vector $E_{\psi}$ in $\mathfrak{g}$, which belongs to $\mathcal{O}_{\min}$. 
The above result of D.~Garfinkle asserts that 
$\gr_2\Omega_{\mathfrak{k}_{\C}}$ is zero at $E_{\psi}$ via the Killing form. 
Let $\{X_i\}_{1\le i\le \dim{\mathfrak{k}}}$ be an orthonormal basis of $\mathfrak{k}$ with respect to the negative of the Killing form. 
By $\gr_2\Omega_{\mathfrak{k}_{\C}}=-\sum_{i=1}^{\dim\mathfrak{k}}X_i^2$, we see $E_{\psi}\in \mathfrak{p}$ and $\psi$ is an imaginary root, which contradicts $\psi\neq 0$.

Therefore $\gr_2\Omega_{\mathfrak{k}_{\C}}$ is nonzero in $(S(\mathfrak{g}_{\C})/\gr J_0)^K$ and $n_0$ is one or two. 
We see that $n_0$ is one exactly when $W(\psi)\cong \mathfrak{g}_{\C}$ has a nontrivial $K$-fixed vector, or $G/K$ is Hermitian. 
Let $T'$ be any nonzero central element in $\mathfrak{k}_{\C}$ when $G/K$ is Hermitian, and $\gr_2\Omega_{\mathfrak{k}_{\C}}$ when $G/K$ is non-Hermitian. 
Since $\gr J_0$ is prime, the subalgebra of $(S(\mathfrak{g}_{\C}/\gr J_0))^K$ generated by $T'$ is the polynomial algebra in $T'$, which proves (\ref{sym-version}).

We next suppose that $\mathfrak{g}_{\C}$ is not simple. 
Since the irreducible $\mathfrak{g}$-module $W(n\psi)$ with highest weight $n\psi$ is self-dual, we have 
\begin{align*}
A&\cong (U(\mathfrak{g})/J_0\otimes U(\mathfrak{g})/J_0)^{\diag\mathfrak{g}}
\cong (S(\mathfrak{g})/\gr J_0\otimes S(\mathfrak{g})/\gr J_0)^{\diag\mathfrak{g}}\\
&\cong \bigoplus_{m,n\ge 0}(W(m\psi)\otimes W(n\psi))^{\diag\mathfrak{g}}
= \bigoplus_{n\ge 0}(W(n\psi)\otimes W(n\psi))^{\diag\mathfrak{g}}
\end{align*}
and $(W(n\psi)\otimes W(n\psi))^{\diag\mathfrak{g}}$ is one-dimensional for all $n\ge 0$. 
Then $(S(\mathfrak{g})/\gr J_0\otimes S(\mathfrak{g})/\gr J_0)^{\diag\mathfrak{g}}$ admits a structure of graded algebra by letting $(W(n\psi)\otimes W(n\psi))^{\diag\mathfrak{g}}$ the space of homogeneous vectors of degree $n$. 

Let $\{Y_j\}_{1\le j\le\dim{\mathfrak{g}}}$ be an basis of $\mathfrak{g}$, $\{Y^j\}_{1\le j\le\dim{\mathfrak{g}}}$ the dual basis with respect to the Killing form, and $\Omega_{\mathfrak{g}}$ the Casimir element of $\mathfrak{g}$. 
Then $\Omega_{\diag\mathfrak{g}}$ is written as $\Omega_{\mathfrak{g}}\otimes 1 +1\otimes \Omega_{\mathfrak{g}}+\sum_j(Y_j\otimes Y^j + Y^j\otimes Y_j)$. 
As $J_0$ has an infinitesimal character, 
the element in $(S(\mathfrak{g})/\gr J_0\otimes S(\mathfrak{g})/\gr J_0)^{\diag\mathfrak{g}}$ corresponding to $\Omega_{\diag\mathfrak{g}}$ 
belongs to $\C+(\mathfrak{g}\otimes \mathfrak{g})^{\diag\mathfrak{g}}$ 
and is not a scalar. 

Since $S(\mathfrak{g})/\gr J_0$ is an integral domain, so is $S(\mathfrak{g})/\gr J_0\otimes S(\mathfrak{g})/\gr J_0$. 
Therefore the subalgebra generated by $\Omega_{\diag\mathfrak{g}}$ is the polynomial algebra in $\Omega_{\diag\mathfrak{g}}$. 
The above argument shows $(S(\mathfrak{g})/\gr J_0)^K=\C[\Omega_{\diag\mathfrak{g}}]$, which completes the proof of (\ref{K-inv2}). 
\end{proof}
\begin{proof}[Proof of Proposition \ref{criterion}]
If $\mathfrak{g}$ is a complex Lie algebra, we can take an isomorphism $\mathfrak{g}_{\C}\cong \mathfrak{g}\oplus\mathfrak{g}$ which gives the one between $\mathfrak{k}_{\C}$ and $\diag\mathfrak{g}$. 
Under the induced isomorphism $U(\mathfrak{g}_{\C})\cong U(\mathfrak{g})\otimes U(\mathfrak{g})$, the Casimir element $\Omega_{\mathfrak{k}_{\C}}$ corresponds to $\Omega_{\diag\mathfrak{g}}$. 

Since any central element of $\mathfrak{k}$ and the Casimir element $\Omega_{\mathfrak{k}_{\C}}$ act on each $K$-isotypic component of a minimal $(\mathfrak{g}_{\C},K)$-module by a scalar multiple, the same assertion holds for any element of $U(\mathfrak{g}_{\C})^K$ by Lemma \ref{K-inv}. 
This completes the proof by the $K$-semisimplicity of $(\mathfrak{g}_{\C},K)$-modules and the following fact. 
\end{proof}
\begin{fact}[{\cite[Theorem 5.5]{LM73}}]\label{LM}
Let $\tau$ be an irreducible finite-dimensional representation of $K$. 
\begin{enumerate}
\item Let $\pi$ be an irreducible $(\mathfrak{g}_{\C},K)$-module. 
	If the $U(\mathfrak{g}_{\C})^K$-module $\Hom_K(\tau,\pi)$ is nonzero, 
	then it is irreducible. 
\item Let $\pi_1, \pi_2$ be irreducible $(\mathfrak{g}_{\C},K)$-modules. 
	If the $U(\mathfrak{g}_{\C})^K$-module $\Hom_K(\tau,\pi_1)$ is nonzero and 
	isomorphic to $\Hom_K(\tau,\pi_2)$, 
	then $\pi_1$ and $\pi_2$ are isomorphic as $(\mathfrak{g}_{\C},K)$-modules. 
\end{enumerate}
\end{fact}
\begin{rem}\label{rem:non-exist}
Lemma \ref{K-inv} $(\ref{K-inv1})$ gives another proof of the well-known fact: when $\mathfrak{g}_{\C}$ is simple and $\mathcal{O}_{\min}\cap\mathfrak{p}_{\C}=\emptyset$, there exists no minimal $(\mathfrak{g}_{\C},K)$-module. 
If existed, the center $Z(\mathfrak{k}_{\C})$ of $U(\mathfrak{k}_{\C})$ would act by scalars by Lemma \ref{K-inv} $(\ref{K-inv1})$ and therefore only one $K$-type would occur. 
It contradicts the infinite-dimensionality of minimal $(\mathfrak{g}_{\C},K)$-modules because they are $K$-multiplicity free.
\end{rem}
\section{classification}\label{section:classification}
Let $G$ be a connected simply connected simple real Lie group not of type $A$, and $\mathfrak{g}, \mathfrak{k}, \mathfrak{p}, K, J_0$ as in Section $\ref{section:preliminaries}$. 
We recall that $K$ is not compact exactly when $G/K$ is Hermitian, and that $(\mathfrak{g}_{\C},K)$-modules in this paper are $K$-semisimple (see Definition $\ref{df:g,K}$). 
In this section, we give the classification of minimal $(\mathfrak{g}_{\C},K)$-modules. 

We fix a Cartan subalgebra $\mathfrak{t}$ of $\mathfrak{k}$. 
Let $\mathfrak{k}_{ss}:=[\mathfrak{k},\mathfrak{k}], \mathfrak{t}_{ss}:=\mathfrak{k}_{ss}\cap\mathfrak{t}$ and $\mathfrak{z}(\mathfrak{k})$ the center of $\mathfrak{k}$. 
We will regard the dual $\mathfrak{t}^{\vee}$ as $\mathfrak{z}(\mathfrak{k})^{\vee}+\mathfrak{t}_{ss}^{\vee}$ via the decomposition $\mathfrak{t}=\mathfrak{z}(\mathfrak{k})\oplus\mathfrak{t}_{ss}$, and write $\mu_{ss}$ for 
the semisimple part of $\mu\in\mathfrak{t}_{\C}^{\vee}$.
Choose a positive system $\Delta^+=\Delta^+(\mathfrak{k}_{ss}, \mathfrak{t}_{ss})$.
When $\mu\in\mathfrak{t}_{\C}^{\vee}$ and the restriction $\mu_{ss}$ is dominant integral, we will write $V(\mu)$ for the irreducible representation of $K$ with highest weight $\mu$. 

Let us state the main theorem in this article. 
\begin{thm}\label{main-thm}
Let $G$ be a connected simply connected simple real Lie group not of type $A$.
	\begin{enumerate}
	\item\label{K-type}
		For any minimal $(\mathfrak{g}_{\C},K)$-module, 
		there exist some $\mu_0\in\mathfrak{t}_{\C}^{\vee}$ and 
		a highest weight $\beta\in\mathfrak{t}_{\C}^{\vee}$ 
		of the $K$-module $\mathfrak{p}_{\C}$ such that 
		the restriction $(\mu_0)_{ss}$ is dominant integral and 
		the $K$-type decomposition is written as 
		\begin{equation}\label{main-K-type}
		\bigoplus_{n\ge 0}V(\mu_0+n\beta).
		\end{equation}
	\item \label{main-real}
		When $\mathfrak{g}_{\C}$ is simple, 
		all minimal $K$-types $V(\mu_0)$ of minimal $(\mathfrak{g}_{\C},K)$-modules are 
		given in Tables $\ref{classification1}$ and $\ref{classification2}$. 
		In particular, the action of $K$ on each minimal $(\mathfrak{g}_{\C},K)$-module descends to 
		a compact group covered by $K$, 
		and the following hold. 
		\begin{enumerate}
		\item When $\mathfrak{g}\cong \mathfrak{sp}(n,\R) (n\ge 2)$, 
			there are exactly four minimal $(\mathfrak{g}_{\C},K)$-modules. 
			They consist of the two irreducible components of the underlying $(\mathfrak{g}_{\C}, K)$-module of the Weil representation 
			and their complex conjugates. 
			They are highest or lowest weight modules. 
		\item When $G/K$ is Hermitian and $\mathfrak{g}$ is not of type $C$, 
			there are exactly two minimal $(\mathfrak{g}_{\C},K)$-modules. 
			They are highest or lowest weight modules 
			where the minimal $K$-types $V(\mu_0)$ are one-dimensional, 
			and they are the complex conjugates of each other. 
		\item When $\mathfrak{g}$ is compact or isomorphic to 
			$\mathfrak{so}(n,1) (n\ge 6), \mathfrak{sp}(p,q) (p,q\ge 1), \e_{6(-26)}, \f_{4(-20)}$ 
			or $\mathfrak{so}(p,q) (p, q\ge 4, p+q \text{ is odd})$, 
			there are no minimal $(\mathfrak{g}_{\C},K)$-modules. 
		\item When $\mathfrak{g}$ is isomorphic to 
			$\mathfrak{so}(p,q) (p, q \ge 3, p+q\ge 8, p+q \text{ is even}), 
			\mathfrak{so}(p,3) (p\ge 4 \text{ is even}), \e_{6(6)}, \e_{6(2)}, \e_{7(7)}, 
			\e_{7(-5)}, \e_{8(8)}, \e_{8(-24)}, \f_{4(4)}$ or $\g_{2(2)}$, 
			there is exactly one minimal $(\mathfrak{g}_{\C},K)$-module. 
		\end{enumerate}
	\item \label{main-cpx}
		When $\mathfrak{g}= \mathfrak{sp}(n,\C) (n\ge 2)$, 
		there are exactly two minimal $(\mathfrak{g}_{\C},K)$-modules. 
		They are the restrictions of the irreducible components of the underlying $(\mathfrak{g}_{\C}, K)$-module of 
		the Weil representation for $\mathfrak{sp}(2n,\R)$ to $\mathfrak{sp}(n,\C)$, where the imbedding is induced by forgetting the complex structure of an $n$-dimensional complex simplectic vector space. 
		One has nonzero $K$-fixed vectors 
		and is isomorphic to the Harish-Chandra bimodule $U(\mathfrak{g}_{\C})/J_0$ (see Remark $\ref{bimodule}(\ref{bimodule2})$). 
		The minimal $\Sp(n)$-type of the other is the $(2n)$-dimensional representation defining $\Sp(n)$. 
		
		When $\mathfrak{g}$ is a complex Lie algebra not of type $C$, 
		there is exactly one minimal $(\mathfrak{g}_{\C},K)$-module. 
		It is isomorphic to the Harish-Chandra bimodule $U(\mathfrak{g}_{\C})/J_0$. 
	\end{enumerate}
\end{thm}
	\begin{table}[ht]
	\begin{center}
	\begin{threeparttable}
	\caption{All minimal $(\mathfrak{g}_{\C},K)$-modules when $G/K$ is Hermitian}
	\label{classification1}
	\begin{tabular}{cccc}\toprule
		$\mathfrak{g}$&$K$&$\mathfrak{p}_{\C}=\mathfrak{p}_+\oplus\mathfrak{p}_-$&minimal $K$-type $V(\mu_0)$\\\hline\hline
		$\mathfrak{so}(p,2)$&$\Spin(p)\times\R$&$\C^{p}\boxtimes(\C_{1}\oplus \C_{-1})$
		&$\C\boxtimes\C_{\pm(p-2)/2}$\\
		$(p\ge 5)$&&&\\\hline
	    $\mathfrak{sp}(n,\R)$&$\SU(n)\times\R$&$S^2(\C^n)\boxtimes\C_{1}$
	    &$\C\boxtimes\C_{\pm n/4}$,\\
		$(n\ge 2)$&&$\oplus S^2(\wedge^{n-1}\C^n)\boxtimes\C_{-1}$
		&$\C^n\boxtimes\C_{(n+2)/4}$,\\
		&&&$\wedge^{n-1}\C^n\boxtimes\C_{-(n+2)/4}$\\\hline
		$\mathfrak{so}^*(2n)$&$\SU(n)\times\R$&
		$\wedge^2\C^n\boxtimes\C_{1}$&$\C\boxtimes\C_{\pm n/2}$\\
		$(n\ge 4)$&&$\oplus\wedge^{n-2}\C^n\boxtimes\C_{-1}$&\\\hline
		$\e_{6(-14)}$&$\Spin(10)\times\R$&
		$V(\omega_5)\boxtimes(\C_{1}\oplus\C_{-1})$\tnote{*a}
		&$\C\boxtimes\C_{\pm4}$\\\hline
		$\e_{7(-25)}$&$\E_6\times\R$&
		$V(\omega_1)\boxtimes\C_{1}$
		&$\C\boxtimes\C_{\pm6}$\\
		&&$\oplus V(\omega_6)\boxtimes\C_{-1}$\tnote{*b}&\\\bottomrule
	\end{tabular}
	\small
    \begin{tablenotes}
      \item[*] We follow the notation of Bourbaki \cite[pp.250--274]{Bou68} for the fundamental weights $\omega_i$ and the standard basis $e_i$. 
      \item[*]
      $\C^n$ denotes the $n$-dimensional representation defining one of $\SO(n)$ or $\SU(n)$, and $\C$ does the trivial representation.
	  \item[*] 
	  The subscripts of $\C$ denote the scalars by which some central element of $\mathfrak{k}_{\C}$ acts. 
      When a subscript on the column of minimal $K$-type is positive (resp. negative), 
      the corresponding minimal $(\mathfrak{g}_{\C},K)$-module has a nonzero $\mathfrak{p}_-$-null (resp. $\mathfrak{p}_+$-null) vector. 
      \item[*a] $\omega_5:=\frac{1}{2}\sum_{i=1}^5e_i$.
      \item[*b] 
      $\omega_1:=\frac{2}{3}(-e_6-e_7+e_8)$, $\omega_6:=e_5+\frac{1}{3}(-e_6-e_7+e_8)$. 
	\end{tablenotes}
	\end{threeparttable}
    \end{center}
    \end{table}
    \begin{table}[ht]
	\begin{center}
	\begin{threeparttable}
	\caption{All minimal $(\mathfrak{g}_{\C},K)$-modules when $\mathfrak{g}_{\C}$ is simple and $G/K$ is non-Hermitian}
	\label{classification2}
	\begin{tabular}{cccc}\toprule
		$\mathfrak{g}$&$K$&$\mathfrak{p}_{\C}$&minimal $K$-type $V(\mu_0)$\\\hline\hline
		$\mathfrak{so}(p,q)$&$\Spin(p)\times\Spin(q)$&$\C^p\boxtimes\C^q$&
		$\C\boxtimes S^{\frac{p-q}{2}}(\C^q)$\\
		\multicolumn{2}{l}{$(p\ge q\ge 3, p+q\ge 8 \text{ even})$}&&\\\hline
		$\mathfrak{so}(p,3)$&$\Spin(p)\times\SU(2)$&
		$\C^p\boxtimes S^2(\C^2)$&$\C\boxtimes S^{p-3}(\C^2)$\\
		$(p\ge 4 \text{ even})$&&&\\\hline
		$\f_{4(4)}$&$\Sp(3)\times\SU(2)$&
		$(\wedge^3\C^6)_0\boxtimes\C^2$\tnote{*a}&$\C\boxtimes S^1(\C^2)$\\\hline
		$\e_{6(2)}$&$\SU(6)\times\SU(2)$&
		$\wedge^3\C^6\boxtimes\C^2$&$\C\boxtimes S^2(\C^2)$\\\hline
		$\e_{7(-5)}$&$\Spin(12)\times\SU(2)$&
		$V(\omega_6)\boxtimes\C^2$\tnote{*b}&$\C\boxtimes S^4(\C^2)$\\\hline
		$\e_{8(-24)}$&$\E_7\times\SU(2)$&
		$V(\omega_7)\boxtimes \C^2\tnote{*c}$&$\C\boxtimes S^8(\C^2)$\\\hline
		$\g_{2(2)}$&$\SU(2)_{\text{short}}
		\times\SU(2)_{\text{long}}$\tnote{*d}
		&
		$S^3(\C^2)\boxtimes \C^2$&$S^2(\C^2)\boxtimes \C$\\\hline
		$\e_{6(6)}$&$\Sp(4)$&$(\wedge^4\C^8)_0$\tnote{*a}&$\C$\\\hline
		$\e_{7(7)}$&$\SU(8)$&$\wedge^4\C^8$&$\C$\\\hline
		$\e_{8(8)}$&$\Spin(16)$&$V(\omega_8)$\tnote{*e}&$\C$\\\bottomrule
	\end{tabular}
	\small
	\begin{tablenotes}
	  \item[*a] $(\bigwedge^3\C^6)_0, (\bigwedge^4\C^8)_0$ denote 
	  the irreducible $\SU(6),\Sp(4)$-module with highest weight $\sum_{i=1}^{3}e_i, \sum_{i=1}^{4}e_i$ respectively. 
	  \item[*b] $\omega_6:=\frac{1}{2}\sum_{i=1}^{6}e_i$.
	  \item[*c] $\omega_7:=e_6+\frac{1}{2}(-e_7+e_8)$. 
	  \item[*d] The first (resp. second) factor of $\mathfrak{k}_{\C}$ has a short (resp. long) root vector with respect to a compact Cartan subalgebra of $\mathfrak{g}$. 
	  \item[*e] $\omega_8:=\frac{1}{2}\sum_{i=1}^{8}e_i$.
	\end{tablenotes}
	\end{threeparttable}
    \end{center}
    \end{table}

\begin{rem}
Let $K'$ be a connected Lie group covered by $K$. 
Minimal $(\mathfrak{g}_{\C},K')$-modules are exactly minimal $(\mathfrak{g}_{\C},K)$-modules 
where the $K$-action on the minimal $K$-type $V(\mu_0)$ descends to $K'$. 
\end{rem}

\begin{rem}\label{Vog-K-type}
\begin{enumerate}
\item 
D.~Vogan \cite[Corollary 3.6]{Vog81} proved that if the condition (i) of Fact \ref{non-exist} holds and the Gelfand--Kirillov dimension of infinite-dimensional irreducible $(\mathfrak{g}_{\C},K)$-module $\pi$ is half the complex dimension of the minimal nilpotent orbit, 
then the $K$-type decomposition of $\pi$ is written as 
\[\bigoplus_{1\le i\le l, n\in\N} V(\mu_i+n\beta)\] 
where $\beta$ is a highest weight  of $\mathfrak{p}_{\C}$ and $\mu_1, \cdots, \mu_l$ are elements of $\mathfrak{t}_{\C}^{\vee}$ such that $(\mu_1)_{ss}, \cdots, (\mu_l)_{ss}$ are dominant integral. 
It was also shown that there exists minimal $(\mathfrak{g}_{\C},K)$-module whose $K$-types are of the form (\ref{main-K-type}) for some $\mathfrak{g}$ \cite[Theorem 2.11]{Vog81}. 
\item By Theorem \ref{main-thm}, it follows that R.~Brylinski and B.~Kostant \cite[Theorem 4]{BK94M} constructed all minimal representations when $\mathfrak{g}_{\C}$ is simple and $G/K$ is non-Hermitian, 
and P.~Torasso \cite{Tor97} constructed all those when 
$\R$-$\rank(\mathfrak{g})\ge 3$. 
\end{enumerate}
\end{rem}

The rest of this section is devoted to the proof of Theorem \ref{main-thm}. 

The assertion about the non-existence of minimal $(\mathfrak{g}_{\C},K)$-modules follows from that in Section \ref{section:intro} and Fact \ref{non-exist} (we see $\mathfrak{so}(4,1)\cong\mathfrak{sp}(1,1), \mathfrak{so}(5,1)\cong\mathfrak{sl}(2,\Ha)$). 
We refer the reader to \cite{Fol89} for the construction and the $K$-type structure of the Weil representation, 
to \cite[Section 2.3.1]{HKM14} for minimality of the two irreducible components of the underlying $(\mathfrak{g}_{\C}, K)$-module of it, for example. 
Then it follows easily that the irreducible components for $\mathfrak{g}=\mathfrak{sp}(2n,\R)$ are still irreducible and minimal when they are restricted to the subalgebra $\mathfrak{sp}(n,\C)$. 
As for the existence of the other minimal $(\mathfrak{g}_{\C},K)$-modules whose $K$-types are given in Theorem \ref{main-thm},  
see A. Braverman and A. Joseph \cite[Proposition 6.6]{BJ98} 
(there is some typo: $\dim\mathbf{G}e_{\beta} -1$ should be $\dim\mathbf{G}e_{\beta} -2$) 
for the case $\mathfrak{g}_{\C}$ is simple, $G/K$ is Hermitian and $V(\mu_0)$ is one-dimensional, 
R. Brylinski and B. Kostant \cite{BK94M} when $G/K$ is non-Hermitian, and 
D. Garfinkle \cite[Chapter I\hspace{-.1em}I\hspace{-.1em}I and V]{Gar82} for the case $\mathfrak{g}$ is a complex Lie algebra and $V(\mu_0)$ is the trivial representation. 

By Proposition \ref{criterion}, minimal $(\mathfrak{g}_{\C},K)$-modules are $K$-multiplicity free and two minimal $(\mathfrak{g}_{\C},K)$-modules are isomorphic if and only if they have a common $K$-type. 
Therefore, it suffices to show the following. 

\begin{thm}\label{same-K}
For any minimal $(\mathfrak{g}_{\C},K)$-module $\pi$, there exists a minimal $(\mathfrak{g}_{\C},K)$-module which appears in Theorem $\ref{main-thm}$ $(\ref{main-real})$, $(\ref{main-cpx})$ and has a common $K$-type with $\pi$. 
\end{thm}
\begin{proof}
Let $\pi$ be a minimal $(\mathfrak{g}_{\C},K)$-module. 
The main idea of the proof is to narrow down possible $K$-types in $\pi$ by using the $K$-types in already known minimal $(\mathfrak{g}_{\C},K)$-modules.

We may assume that $\mathfrak{g}$ does not satisfy the condition (iii) in Fact \ref{non-exist}. 
Let $W$ be the Weyl group with respect to $(\mathfrak{k},\mathfrak{t})$ and 
\begin{center}
$\rho_{\mathfrak{k}_{ss}}$ half the sum of positive roots for $(\mathfrak{k}_{ss},\mathfrak{t}_{ss})$. 
\end{center}
Take a minimal $(\mathfrak{g}_{\C},K)$-module $\pi_0\cong \bigoplus_{n\ge 0} V(\mu_0+n\beta)$ appearing in Theorem \ref{main-thm}. 
Then the highest weights of $K$-types in $\pi$ must lie in the $W$-translates of a line:
\begin{lem}\label{line-inclusion}
Let $\mu\in\mathfrak{t}_{\C}^{\vee}$. Assume that the restriction $\mu_{ss}$ is dominant integral. 
If the $K$-type $V(\mu)$ occurs in $\pi$, then the weight $\mu+\rho_{\mathfrak{k}_{ss}}$ belongs to $W(\mu_0+\rho_{\mathfrak{k}_{ss}}+\R\beta)$.
\end{lem}
\begin{proof}
Let $Z(\mathfrak{k}_{\C})$ be the center of $U(\mathfrak{k}_{\C})$ and 
$$\gamma:Z(\mathfrak{k}_{\C})\xrightarrow{\cong} \mathcal{P}(\mathfrak{t}_{\C}^{\vee})^W$$ 
the Harish-Chandra isomorphism. 
Here $\mathcal{P}(\mathfrak{t}_{\C}^{\vee})$ denotes the algebra of polynomials on $\mathfrak{t}_{\C}^{\vee}$. 
The infinitesimal character of $V(\mu)$ is the $W$-orbit through $\mu+\rho_{\mathfrak{k}_{ss}}$. 
Since the Zariski closure of $\N\beta$ is the line $\C\beta$, 
we have
\begin{align*}
\gamma(Z(\mathfrak{k}_{\C})\cap \Ann\pi)
=\gamma(Z(\mathfrak{k}_{\C})\cap \Ann\pi_0)
=\mathcal{I}(\mu_0+\rho_{\mathfrak{k}_{ss}}+\N\beta)^W
=\mathcal{I}(\mu_0+\rho_{\mathfrak{k}_{ss}}+\C\beta)^W.
\end{align*}
Here $\mathcal{I}(\Lambda)$ denotes the ideal of $\mathcal{P}(\mathfrak{t}_{\C}^{\vee})$ defined by a subset $\Lambda$ of $\mathfrak{t}_{\C}^{\vee}$. 
Therefore, when the $K$-type $V(\mu)$ occurs in the $K$-type of $\pi$, we see
\begin{equation}\label{in-orbits}
\mathcal{I}(\mu_0+\rho_{\mathfrak{k}_{ss}}+\C\beta)^W\subset \mathcal{I}(\{\mu+\rho_{\mathfrak{k}_{ss}}\})^W.
\end{equation}

Since $\im\mathfrak{t}_{ss}^{\vee}$ is invariant under the action of $W$, it suffices to show $\mu+\rho_{\mathfrak{k}_{ss}}\in W(\mu_0+\rho_{\mathfrak{k}_{ss}}+\C\beta)$. 
On the contrary, we assume $\mu+\rho_{\mathfrak{k}_{ss}}\notin W(\mu_0+\rho_{\mathfrak{k}_{ss}}+\C\beta)$. 
Then we can take a polynomial $f$ on $\mathfrak{t}_{\C}^{\vee}$ such that $f(\mu_0+\rho_{\mathfrak{k}_{ss}}+\C\beta)=\{0\}$ and $f(w(\mu+\rho_{\mathfrak{k}_{ss}}))\neq 0$ for all $w\in W$. 
While the polynomial $\widetilde{f}(\lambda):=\prod_{w\in W}f(w\lambda)$ belongs to $\mathcal{I}(\mu_0+\rho_{\mathfrak{k}_{ss}}+\C\beta)^W$, 
it does not to $\mathcal{I}(\{\mu+\rho_{\mathfrak{k}_{ss}}\})^W$, which contradicts (\ref{in-orbits}). 
\end{proof}

We next consider which $W$-translates of the line $(\mu_0)_{ss}+\rho_{\mathfrak{k}_{ss}}+\R\beta_{ss}$ can contain infinite number of dominant integral weights for $(\mathfrak{k}_{ss}, \mathfrak{t}_{ss})$. 
Let us denote by $(\cdot,\cdot)$ the inner product  on $\im\mathfrak{t}_{ss}^{\vee}$ given by the negative of the Killing form of $\mathfrak{k}_{ss}$, 
by $L_+$ the set of dominant integral weights for $(\mathfrak{k}_{ss},\mathfrak{t}_{ss})$, and 
by $(\R\beta_{ss})^{\perp}$ the orthogonal complement of $\R\beta_{ss}$ in $\mathfrak{t}_{ss}$. 
Let 
\begin{align*}
\xi_0:=\text{the $(\R\beta_{ss})^{\perp}$-component of $(\mu_0)_{ss}+\rho_{\mathfrak{k}_{ss}}$} 
\end{align*}
and $\Delta_{\lambda_{ss}}\subset \im\mathfrak{t}_{ss}^{\vee}$ 
(resp. $\Delta_{\lambda_{ss}}^+$) be the set of roots (resp. positive roots) for $(\mathfrak{k}_{ss},\mathfrak{t}_{ss})$ orthogonal to $\lambda_{ss}\in \im\mathfrak{t}_{ss}^{\vee}$, and $W_{\beta_{ss}}$ the Weyl group of $\Delta_{\beta_{ss}}$. 
The longest element of $W$ with respect to the positive system and that of $W_{\beta_{ss}}$ are denoted by $w_l$ and $w_{\beta_{ss},l}$ respectively. 

\begin{lem}\label{two-lines}
Let 
\begin{align*}
w_0:=w_lw_{\beta_{ss},l}.
\end{align*} 
Then $w_0$ is the unique nontrivial element in $W$ such that the cardinality of $w_0((\mu_0)_{ss}+\rho_{\mathfrak{k}_{ss}}+\R\beta_{ss})\cap L_+$ is infinity.
\end{lem}
\begin{proof}
Let $w\in W$. 
By $\rho_{\mathfrak{k}_{ss}}\in L_+$, we see that $w((\mu_0)_{ss}+\rho_{\mathfrak{k}_{ss}}+\R\beta_{ss})\cap L_+$ is an infinite set if and only if $w\beta_{ss}$ belongs to either $L_+$ or its negative $-L_+$ and 
\begin{align}\label{xi0-positivity}
(\alpha, w\xi_0)\ge 0\text{ for any positive root $\alpha$ with $(\alpha, w\beta_{ss})=0$}.
\end{align}

Let us first assume $w\beta_{ss}\in L_+$. 
Then we have $w\beta_{ss}=\beta_{ss}$ and $w\in W_{\beta_{ss}}$ by $\beta_{ss}\in L_+$ (see \cite[Lemma 10.3.B]{Hum78} for instance). 
Since $(\alpha,\xi_0)=(\alpha, (\mu_0)_{ss}+\rho_{\mathfrak{k}_{ss}})>0$ for all $\alpha\in\Delta^+_{\beta_{ss}}$, the $(\R\beta_{ss})^{\perp}$-component $\xi_0$ belongs to the dominant Weyl chamber for $\Delta_{\beta_{ss}}^+$. 
The Weyl group acts simply transitively on the set of Weyl chambers, hence the condition (\ref{xi0-positivity}) holds if and only if $w$ is the identity element. 

We next suppose $w\beta_{ss}\in -L_+$. 
In the same manner we see $w\beta_{ss}=w_l\beta_{ss}$, and $w_l^{-1}w=w_lw\in W_{\beta_{ss}}$. 
Since $\xi_0$ belongs to the dominant Weyl chamber for $\Delta_{\beta_{ss}}^+$, 
the condition (\ref{xi0-positivity}) is equivalent to $\Delta_{w\beta_{ss}}^+=\Delta_{w\beta_{ss}}\cap w\Delta^+$. 
Since $w_lw$ fixes $\beta_{ss}$, the latter is also equivalent to $\Delta_{\beta_{ss}}-\Delta^+=\Delta_{\beta_{ss}}\cap w_lw\Delta^+$, or $w_lw=w_{\beta_{ss},l}$, which is the desired conclusion. 
\end{proof}

Since $\pi$ is infinite-dimensional and $W$ is a finite group, there exists $\mu\in \mathfrak{t}_{\C}^{\vee}$ such that the restriction $\mu_{ss}$ is a dominant integral weight, the norm $\|\mu_{ss}\|$ is sufficiently large, the $K$-type $V(\mu)$ occurs in $\pi$, and the weight $\mu+\rho_{\mathfrak{k}_{ss}}$ belongs to $\mu_0+\rho_{\mathfrak{k}_{ss}}+\R\beta$ or $w_0(\mu_0+\rho_{\mathfrak{k}_{ss}}+\R\beta)$ by Proposition \ref{criterion} and Lemmas \ref{line-inclusion} and \ref{two-lines}.

When $G/K$ is non-Hermitian, the two lines above coincide: 
\begin{prop}\label{same-line}
Assume that $G/K$ is non-Hermitian and take $w_0\in W$ as in Lemma $\ref{two-lines}$.
Then we have $\mu_0+\rho_{\mathfrak{k}_{ss}}+\R\beta=w_0(\mu_0+\rho_{\mathfrak{k}_{ss}}+\R\beta)$.
\end{prop}
\begin{proof}
It suffices to show that there exists an element $w_0$ of the Weyl group $W$ satisfying $w_0\beta=-\beta$ and $w_0\xi_0=\xi_0$ by the uniqueness in Lemma \ref{two-lines}. 

We first assume that $\mathfrak{g}$ is a complex Lie algebra. 
Let $J$ be the complex structure of $\mathfrak{g}$. 
By $\mathfrak{p}=J\mathfrak{k}$, the highest weight $\beta$ of $\mathfrak{p}$ is the highest root of $\mathfrak{k}$. 
Then the reflection with respect to $\beta$ has the desired properties. 

When $\mathfrak{g}_{\C}$ is simple and $G/K$ is non-Hermitian, we check that $w_0$ in Table \ref{constants2} satisfies the condition by using Tables \ref{constants1} and \ref{constants2}. 
Here we follow the notation of Bourbaki \cite[pp.250--274]{Bou68} for the standard basis, which is compatible with that in Table \ref{classification2}. 
We omit the verification except for the case 
$\mathfrak{g}\cong\mathfrak{e}_{8(-24)}$, 
which is only the case where $\mathfrak{k}$ has an exceptional simple factor. 
In that case, using the notation in Tables \ref{constants1} and \ref{constants2}, we see that $\xi_0=e_2+2e_3+3e_4+4e_5-4e_6-4e_7+4e_8$ is orthogonal to $\eta_1=\frac{1}{2}(e_1-e_2-e_3+e_4+e_5-e_6+e_7-e_8)$, $\eta_2=\frac{1}{2}(-e_1+e_2+e_3-e_4+e_5-e_6+e_7-e_8)$, $e_5+e_6$ and $f_1-f_2$. 
Hence $\xi_0$ is fixed by $w_0=s(e_5+e_6)s(\eta_2)s(\eta_1)s(f_1-f_2)$. 
Moreover, we have 
\begin{align*}
w_0\beta=
&s(e_5+e_6)s(\eta_2)s(\eta_1)\left(e_6-\frac{1}{2}e_7+\frac{1}{2}e_8, -f_1+f_2\right)\\
&=s(e_5+e_6)s(\eta_2)\left(\frac{1}{2}\left(e_1-e_2-e_3+e_4+e_5+e_6\right), -f_1+f_2\right)\\
&=s(e_5+e_6)\left(e_5+\frac{1}{2}e_7-\frac{1}{2}e_8, -f_1+f_2\right)\\
&=\left(-e_6+\frac{1}{2}e_7-\frac{1}{2}e_8, -f_1+f_2\right)
=-\beta,
\end{align*}
which is our claim. 
\end{proof}

The next lemma gives the integral weights on the line $\R\beta_{ss}$. 
It is proven by a case-by-case check, so we omit the proof. 
\begin{lem}\label{period}
The intersection of the weight lattice for $(\mathfrak{k}_{ss}, \mathfrak{t}_{ss})$ with $\R\beta_{ss}$ is $\Z\beta_{ss}/2$ if $\mathfrak{g}$ is isomorphic to $\mathfrak{sp}(n,\R)$ or $\mathfrak{sp}(n,\C)$, and $\Z\beta_{ss}$ otherwise. 
\end{lem}

Let us finish the proof of Theorem \ref{same-K}. 
Consider the set consisting of $\lambda\in\mathfrak{t}_{\C}^{\vee}$ satisfying the following conditions: $\lambda+\rho_{\mathfrak{k}_{ss}}$ belongs to $\mu_0+\rho_{\mathfrak{k}_{ss}}+\R\beta \cup w_0(\mu_0+\rho_{\mathfrak{k}_{ss}}+\R\beta)$, the restriction $\lambda_{ss}$ belongs to $L_+$, and the $K$-type $V(\lambda)$ does not occur in any minimal $(\mathfrak{g}_{\C},K)$-module appearing in Theorem \ref{main-thm}. 
Then the set is finite by Proposition \ref{same-line} and Lemma \ref{period}.
By the argument just before Proposition \ref{same-line}, the given minimal $(\mathfrak{g}_{\C},K)$-module $\pi$ have a common $K$-type with some minimal $(\mathfrak{g}_{\C},K)$-module appearing in Theorem \ref{main-thm}.
\end{proof}

\section{Corollaries}
In this section, we give some corollaries of Theorem \ref{main-thm}. 

Let $G$ be a connected simply connected simple real Lie group not of type $A$, and $\mathfrak{g}, K$ as in Section $\ref{section:preliminaries}$. 
\begin{cor}\label{unitarizability}
	Any minimal $(\mathfrak{g}_{\C},K)$-module (see Definition $\ref{df:min}$) is unitarizable. 
\end{cor}
\begin{proof}
Since minimal $(\mathfrak{g}_{\C},K)$-modules, which are not necessarily unitarizable, are classified in Theorem \ref{main-thm}, 
it suffices to check unitarizability for each of them. 
It follows from the determination of unitary highest weight modules by T.~Enright, R.~Howe and N.~Wallach \cite{EHW83} or H.~P.~Jacobsen \cite{Jak83} when $\mathfrak{g}_{\C}$ is simple and $G/K$ is Hermitian. 
Unitarizability of minimal $(\mathfrak{g}_{\C}, K)$-modules follows from the determination of the unitary dual by Duflo \cite{Duf79} when $\mathfrak{g}$ is a complex Lie algebra of rank two, and 
from that by D.~Vogan \cite[Proposition 15.2]{Vog94} when $\mathfrak{g}=\mathfrak{g}_{2(2)}$. Unitarizability follows from the construction of minimal representations by D.~Kazhdan and G.~Savin \cite{KS90} when $\mathfrak{g}=\mathfrak{so}(2n,\C), 
\mathfrak{so}(n,n)(n\ge 4), \e_6(\C), \e_{6(6)}, \e_7(\C), \e_{7(7)}, \e_8(\C)$ or $\e_{8(8)}$, and 
from that by P.~Torasso \cite[Theorem 4.12]{Tor97} when the real rank of $\mathfrak{g}$ is not less than three. 
When $\mathfrak{g}$ is a complex Lie algebra not of type $G_2$, $E_6$ or $E_8$, unitarizability of minimal $(\mathfrak{g}_{\C}, K)$-modules is proven by J.-S.~Huang \cite[Theorem 3.2]{Hua95}. 
There are also other proofs of unitarizability of minimal $(\mathfrak{g}_{\C}, K)$-modules by using various geometric models. 
See for example T. Kobayashi and B. \O rsted \cite[Corollary 3.9.2]{KO03i} for $\mathfrak{g}=\mathfrak{o}(p,q)$ $(p,q\ge 2, p+q\ge 8 \text{ even})$, and J. Hilgert, T. Kobayashi and J. M\"{o}llers \cite{HKM14} for the conformal algebras of Jordan algebras. 
\end{proof}

Moreover, Theorem \ref{main-thm} and the construction of each minimal $(\mathfrak{g}_{\C},K)$-module by P.~Torasso describe the restriction of its unitary globalization to minimal parabolic subgroups. 
\begin{cor}[{\cite[Theorem 4.12]{Tor97}}]\label{irreducibility}
Assume that $\mathfrak{g}_{\C}$ is simple not of type $A$ and the real rank of $\mathfrak{g}$ is not less than three. 
Let $Q$ be a minimal parabolic subgroup of $G$. 
Then the unitary globalization of each minimal $(\mathfrak{g}_{\C},K)$-module is still irreducible as a unitary representation of $Q$. 
\end{cor}

In addition to Corollaries \ref{unitarizability} and \ref{irreducibility}, we have the following remark. 
\begin{rem}\label{unitarizability-complete prime}
Let $G$ be the universal cover of $\Sp(n,\R)$ $(n\ge 2)$ or that of the identity component of the indefinite special orthogonal group $\SO(p,q)$ $(p,q\ge 3, p+q\ge 7)$. 
Let $V$ be an irreducible $(\mathfrak{g}_{\C},K)$-module where the associated variety of the annihilator is the closure of the minimal nilpotent orbit $\mathcal{O}_{\min}$ in $\mathfrak{g}_{\C}$. 
Then $V$ is unitarizable if and only if the annihilator of $V$ is completely prime. 

Let us prove the assertion. Irreducible unitary representations of $G$ under consideration whose associated varieties are the closure of the minimal nilpotent orbit $\mathcal{O}_{\min}$ are classified by J.-S.~Huang and J.-S.~Li \cite{HL99} in terms of local theta lifting. 
Since their classification results agree with the unitary globalizations of minimal $(\mathfrak{g}_{\C},K)$-modules, the assertion follows. 
\end{rem}
\bibliographystyle{amsplain}
\Addresses
\bibliography{references}

\providecommand{\bysame}{\leavevmode\hbox to3em{\hrulefill}\thinspace}
\providecommand{\MR}{\relax\ifhmode\unskip\space\fi MR }
\providecommand{\MRhref}[2]{%
  \href{http://www.ams.org/mathscinet-getitem?mr=#1}{#2}
}
\providecommand{\href}[2]{#2}
\begin{thebibliography}{10}

\bibitem{BSZ06}
L.~Barchini, M.~Sepanski, and R.~Zierau, \emph{Positivity of zeta distributions
  and small unitary representations}, Contemp.\ Math. \textbf{398} (2006),
  1--46.

\bibitem{BZ91}
B.~Binegar and R.~Zierau, \emph{Unitarization of a singular representation of
  ${SO}(p,q)$}, Comm.\ Math.\ Phys. \textbf{138} (1991), 245--258.

\bibitem{Bou68}
N.~Bourbaki, \emph{Groupes et alg{\`e}bres de {L}ie, {C}hapitres $4, 5$ et
  $6$}, Hermann, Paris, 1968.

\bibitem{BJ98}
A.~Braverman and A.~Joseph, \emph{The minimal realization from deformation
  theory}, J.\ Algebra \textbf{205} (1998), 13--36.

\bibitem{Bry98}
R.~Brylinski, \emph{Geometric quantization of real minimal nilpotent orbits},
  Differential Geom.\ Appl. \textbf{9} (1998), no.~1--2, 5--58.

\bibitem{BK94D}
R.~Brylinski and B.~Kostant, \emph{Differential operators on conical
  {L}aglangian manifolds}, Lie Theory and Geometry: In Honor of Bertram Kostant
  (J.-L.\ Brylinski, R.~Brylinski, V.~Guillemin, and V.~Kac, eds.), Progr.\
  Math., vol. 123, Birkh{\"a}user Boston, 1994, pp.~65--96.

\bibitem{BK94E}
\bysame, \emph{Minimal reprerentations of ${E}_6$, ${E}_7$, and ${E}_8$ and the
  generalized {C}apelli identity}, Proc.\ Nat.\ Acad.\ Sci.\ U.S.A. \textbf{91}
  (1994), no.~7, 2469--2472.

\bibitem{BK94M}
\bysame, \emph{Minimal representations, geometric quantization, and unitarity},
  Proc.\ Nat.\ Acad.\ Sci.\ U.S.A. \textbf{91} (1994), no.~13, 6026--6029.

\bibitem{BK95}
\bysame, \emph{Lagrangian models of minimal representations of ${E}_6$, ${E}_7$
  and ${E}_8$}, Functional analysis on the eve of the 21st century
  (S.~Gindikin, J.~Lepowsky, and R.~Wilson, eds.), Progr.\ Math., $131$,
  vol.~1, Birkh{\"a}user Boston, 1995, pp.~13--63.

\bibitem{Duf79}
M.~Duflo, \emph{Repr{\'e}sentations unitaires irr{\'e}ductibles des groupes
  simples complexes de rang deux}, Bull. Soc. Math. France \textbf{107} (1979),
  55--96.

\bibitem{EHW83}
T.~Enright, R.~Howe, and N.~Wallach, \emph{A classification of unitary highest
  weight modules}, Representation theory of reductive groups, Progr.\ Math.,
  vol.~40, Birkh{\"a}user Boston, 1983, pp.~97--143.

\bibitem{EPWW85}
T.~Enright, R.~Parthasarathy, N.~Wallach, and J.~Wolf, \emph{Unitary derived
  functor modules with small spectrum}, Acta Math. \textbf{154} (1985),
  no.~1--2, 105--136.

\bibitem{Fol89}
G.~B. Folland, \emph{Harmonic analysis in phase space}, Annals of {M}athematics
  {S}tudies, vol. 122, Princeton University Press, 1989.

\bibitem{GS04}
W.~T. Gan and G.~Savin, \emph{Uniqueness of {J}oseph ideal}, Math.\ Res.\ Lett.
  \textbf{11} (2004), 589--597.

\bibitem{GS05}
\bysame, \emph{On minimal representations definitions and properties},
  Represent.\ Theory \textbf{9} (2005), 46--93.

\bibitem{Gar82}
D.~Garfinkle, \emph{A new construction of the {J}oseph ideal}, Phd. thesis,
  M.I.T., 1982.

\bibitem{Gon82}
A.~B. Goncharov, \emph{Constructions of {W}eil representations of some simple
  {L}ie algebras}, Funktsional.\ Anal.\ i Prilozhen \textbf{16} (1982), no.~2,
  70--71.

\bibitem{GW94}
B.~Gross and N.~Wallach, \emph{A distinguished family of unitary
  representations for the exceptional groups of real rank $=$ $4$}, Lie Theory
  and Geometry: In Honor of Bertram Kostant (J.-L.\ Brylinski, R.~Brylinski,
  V.~Guillemin, and V.~Kac, eds.), Progr.\ Math., vol. 123, Birkh{\"a}user
  Boston, 1994, pp.~289--304.

\bibitem{GW96}
\bysame, \emph{On quaternionic discrete series representations, and their
  continuations}, J.\ Reine Angew.\ Math. \textbf{481} (1996), 73--123.

\bibitem{HKM14}
J.\ Hilgert, T.\ Kobayashi, and J.\ M{\"o}llers, \emph{Minimal representations
  via {B}essel operators}, J.\ Math.\ Soc.\ Japan \textbf{66} (2014), no.~2,
  349--414.

\bibitem{HKMO12}
J.~Hilgert, T.~Kobayashi, J.~M{\"o}llers, and B.~{\O}rsted, \emph{Fock model
  and {S}egal-{B}argmann transform for minimal representations of {H}ermitian
  {L}ie groups}, J.\ Funct.\ Anal. \textbf{263} (2012), no.~11, 3492--3563.

\bibitem{Hua95}
J.-S.\ Huang, \emph{Minimal representations, shared orbits, and dual pair
  correspondences}, Internat.\ Math.\ Res.\ Notices \textbf{1995} (1995),
  no.~6, 309--323.

\bibitem{HL99}
J.-S.\ Huang and J.-S.\ Li, \emph{Unipotent representations attached to
  spherical nilpotent orbits}, Amer.\ J.\ Math. \textbf{121} (1999), no.~3,
  497--517.

\bibitem{Hum78}
J.~E. Humphreys, \emph{Introduction to lie algebras and representation theory},
  Graduate Texts in Mathematics, vol.~9, Springer-Verlag, New York-Berlin,
  1978.

\bibitem{Jak83}
H.~P. Jakobsen, \emph{Hermitian symmetric spaces and their unitary highest
  weight modules}, J.\ Funct.\ Anal. \textbf{52} (1983), no.~3, 385--412.

\bibitem{Jos76}
A.~Joseph, \emph{The minimal orbit in a simple {L}ie algebra and its associated
  maximal ideal}, Ann.\ Sci. {\'E}cole.\ Norm.\ Sup. $(4)$ \textbf{9} (1976),
  no.~1, 1--29.

\bibitem{Kaz90}
D.~Kazhdan, \emph{The minimal representation of ${D}_{4}$}, Operator Algebras,
  Unitary Representations, Enveloping Algebras, and Invariant Theory
  (A.~Connes, M.~Duflo, A.~Joseph, and R.~Rentschler, eds.), Progr.\ Math.,
  vol.~92, Birkh{\"a}user Boston, 1990, pp.~125--158.

\bibitem{KS90}
D.~Kazhdan and G.~Savin, \emph{The smallest representation of simply laced
  groups}, Festschrift in honor of {I}.\ I.\ {P}iatetski-{S}hapiro on the
  occasion of his sixtieth birthday. {P}art I. (Jerusalem) (S.~Gelbert,
  R.~Howe, and P.~Sarnak, eds.), Israel Math.\ Conf.\ Proc., vol.~2, Weizmann
  {S}cience {P}ress of {I}srael, 1990, pp.~209--223.

\bibitem{Kna02}
A.~W. Knapp, \emph{Lie groups beyond an introduction}, second ed., Progr.\
  Math., vol. 140, Birkh{\"a}user, 2002.

\bibitem{Kob11a}
T.~Kobayashi, \emph{Algebraic analysis of minimal representations}, Publ.\
  Res.\ Inst.\ Math.\ Sci. \textbf{47} (2011), no.~2, 585--611.

\bibitem{KM11}
T.~Kobayashi and G.~Mano, \emph{The {S}chr{\"o}dinger model for the minimal
  representation of the indefinite orthogonal group ${O}(p,q)$}, Mem.\ Amer.\
  Math.\ Soc. \textbf{213} (2011), no.~1000.

\bibitem{KO98}
T.~Kobayashi and B.~{\O}rsted, \emph{Conformal geometry and branching laws for
  unitary representatins attached to minimal nilpotent orbits}, C.\ R.\ Acad.\
  Sci.\ Paris S\'{e}r. I Math. \textbf{326} (1998), no.~8, 925--930.

\bibitem{KO03i}
\bysame, \emph{Analysis on the minimal representation of ${O}(p,q)$, {I}.
  {R}ealization via conformal geometry}, Adv.\ Math. \textbf{180} (2003),
  486--512.

\bibitem{KO03ii}
\bysame, \emph{Analysis on the minimal representation of ${O}(p,q)$, {II}.
  {B}ranching laws}, Adv.\ Math. \textbf{180} (2003), 513--550.

\bibitem{KO03iii}
\bysame, \emph{Analysis on the minimal representation of ${O}(p,q)$, {III}.
  {U}ltrahyperbolic equations on $\mathbb{R}^{p-1,q-1}$}, Adv.\ Math.
  \textbf{180} (2003), 551--595.

\bibitem{KO15}
T.~Kobayashi and Y.~Oshima, \emph{Classification of symmetric pairs with
  discretely decomposable restrictions of $(\mathfrak{g}, k)$-modules}, J.\
  Reine Angew.\ Math. \textbf{703} (2015), 201--223.

\bibitem{Kos90}
B.~Kostant, \emph{The vanishing scalar curvature and the minimal unitary
  representation of ${SO}(4,4)$}, Operator Algebras, Unitary Representations,
  Enveloping Algebras, and Invariant Theory (A.~Connes, M.~Duflo, A.~Joseph,
  and R.~Rentschler, eds.), Progr.\ Math., vol.~92, Birkh{\"a}user Boston,
  1990, pp.~85--124.

\bibitem{LM73}
J.~Lepowsky and G.~W. McCollum, \emph{On the determination of irreducible
  modules by restriction to a subalgebra}, Trans.\ Amer.\ Math.\ Soc.
  \textbf{176} (1973), 45--57.

\bibitem{Li00}
J.-S.\ Li, \emph{Minimal representations $\&$ reductive dual pairs},
  Representation theory of Lie groups (Park City, UT, 1998), IAS/Park City
  Math.\ Ser., vol.~8, Amer.\ Math.\ Soc., 2000, pp.~293--340.

\bibitem{LS08}
H.~Y. Loke and G.~Savin, \emph{The smallest representations of nonlinear covers
  of odd orthogonal groups}, Amer.\ J.\ Math. \textbf{130} (2008), no.~3,
  763--797.

\bibitem{Oku15}
T.~Okuda, \emph{Smallest complex nilpotent orbits with real points}, J.\ Lie
  Theory \textbf{25} (2015), no.~2, 507--533.

\bibitem{Sab96}
H.~Sabourin, \emph{Une repr{\'e}sentation unipotente associ{\'e}e {\`a}
  l'orbite minimale: {L}e cas de $so(4,3)$}, J.\ Funct.\ Anal. \textbf{137}
  (1996), no.~2, 394--465.

\bibitem{Sal06}
H.~Salmasian, \emph{Isolatedness of the minimal representation and minimal
  decay of exceptional groups}, Manuscripta Math. \textbf{120} (2006), no.~1,
  39--52.

\bibitem{Tor97}
P.~Torasso, \emph{M{\'e}thode des orbites de {K}irillov-{D}uflo et
  repr{\'e}sentations minimales des groupes simples sur un corps local de
  caract{\'e}ristique nulle}, Duke Math.\ J. \textbf{90} (1997), no.~2,
  261--377.

\bibitem{Vog81}
D.~Vogan, \emph{Singular unitary representations}, Noncommutative harmonic
  analysis and Lie groups, Lecture Notes in Math., vol. 880, Springer, 1981,
  pp.~506--535.

\bibitem{Vog94}
\bysame, \emph{The unitary dual of ${G}_2$}, Invent.\ Math. \textbf{116}
  (1994), no.~1-3, 677--791.

\end{thebibliography}
\begin{landscape}
    \begin{table}[ht]
  	\begin{center}
	\begin{threeparttable}
  	\caption{Data $1$ of minimal $(\mathfrak{g}_{\C},K)$-modules when $\mathfrak{g}_{\C}$ is simple and $G/K$ is non-Hermitian}
	\label{constants1}
  	\begin{tabular}{c|ccc}\toprule
  		$\mathfrak{g}$&$\rho_{\mathfrak{k}_{ss}}$&$\mu_0$&$\beta$\\\hline\hline
  		$\mathfrak{so}(2n,2m)$&$((n-1,n-2,\cdots,0),$&$(0,(n-m,0,\cdots,0))$
		&$((1,0,\cdots,0),(1,0,\cdots,0))$\\
  		$(n\ge m\ge 2)$&\ \ $(m-1,m-2,\cdots,0))$&&\\\hline
  		$\mathfrak{so}(2n+1,2m+1)$&$((n-1/2,n-3/2,\cdots,1/2),$
		&$(0,(n-m,0,\cdots,0))$&$((1,0,\cdots,0),(1,0,\cdots,0))$\\
  		$(n\ge m\ge 1, n+m\ge 3)$&\ \ $(m-1/2,m-3/2,\cdots,1/2))$\\\hline
  		$\mathfrak{so}(2n,3)$&$((n-1,n-2,\cdots,0),(1,-1))$&$(0,(2n-3,-2n+3))$&$((1,0,\cdots,0),(2,-2))$\\
		$(n\ge 2)$&&&\\\hline
  		$\f_{4(4)}$&$((3,2,1),(1,-1))$&$(0,(1,-1))$&$((1,1,1),(1,-1))$\\\hline
  		$\e_{6(2)}$&$((5/2,3/2,\cdots,-5/2),(1,-1))$&$(0,(2,-2))$&$(1/2(1,1,1,-1,-1,-1),(1,-1))$\\\hline
  		$\e_{7(-5)}$&$((5,4,\cdots,0),(1,-1))$&$(0,(4,-4))$&$(1/2(1,1,1,1,1,1),(1,-1))$\\\hline
  		$\e_{8(-24)}$&$((0,1,2,3,4,5,-17/2,17/2),(1,-1))$
		&$(0,(8,-8))$&$((0,0,0,0,0,1,-1/2,1/2),(1,-1))$\\\hline
  		$\g_{2(2)}$&$((1,-1),(1,-1))$&$((2,-2),0)$&$((3,-3),(1,-1))$\\\hline
  		$\e_{6(6)}$&$(4,3,2,1)$&$0$&$(1,1,1,1)$\\\hline
  		$\e_{7(7)}$&$(7/2,5/2,\cdots,-7/2)$&$0$&$1/2(1,1,1,1,-1,-1,-1,-1)$\\\hline
  		$\e_{8(8)}$&$(7,6,\cdots,1,0)$&$0$&$1/2(1,1,\cdots,1)$\\\bottomrule
  	\end{tabular}
	\begin{tablenotes}
	  \item[*] We write $\mu_0$ for the highest weight of the minimal $K$-type of a minimal representation, 
	  and $\beta$ for that of $K$-module $\mathfrak{p}_{\C}$.
	\end{tablenotes}
	\end{threeparttable}
    \end{center}
    \end{table}
\end{landscape}
\begin{landscape}
    \begin{table}[htb]
    \begin{center}
    \begin{threeparttable}
	\caption{Data $2$ of minimal $(\mathfrak{g}_{\C},K)$-modules when $\mathfrak{g}_{\C}$ is simple and $G/K$ is non-Hermitian}
	\label{constants2}
    \begin{tabular}{c|cc}\toprule
  		$\mathfrak{g}$&$\xi_0$&$w_0$\\\hline\hline
  		$\mathfrak{so}(2n,2m)$&$((0,n-2,n-3,\cdots,0),$
		&$s(e_1+e_n)s(e_1-e_n)s(f_1+f_m)s(f_1-f_m)$
		\\
  		$(n\ge m\ge 2)$&\ \ $(0,m-2,m-3,\cdots,0))$\\\hline
  		$\mathfrak{so}(2n+1,2m+1)$&$((0,n-3/2,n-5/2,\cdots,1/2),$&$s(e_1)s(f_1)$\\
  		$(n\ge m\ge 1, n+m\ge 3)$&\ \ $(0,m-3/2,m-5/2,\cdots,1/2))$\\\hline
  		$\mathfrak{so}(2n,3)$&$((0,n-2,\cdots,0),0)$&$s(e_1+e_n)s(e_1-e_n)s(f_1-f_2)$\\
		$(n\ge 2)$&&\\\hline
  		$\f_{4(4)}$&$((1,0,-1),0)$&$s(e_1+e_3)s(e_2)s(f_1-f_2)$\\\hline
  		$\e_{6(2)}$&$((1,0,-1,1,0,-1),0)$&$s(e_1-e_4)s(e_2-e_5)s(e_3-e_6)s(f_1-f_2)$\\\hline
  		$\e_{7(-5)}$&$((5/2,3/2,\cdots,-5/2),0)$&$s(e_1+e_6)s(e_2+e_5)s(e_3+e_4)s(f_1-f_2)$\\\hline
  		$\e_{8(-24)}$&$((0,1,2,3,4,-4,-4,4),0)$&$s(e_5+e_6)s(\eta_2)s(\eta_1)s(f_1-f_2)$\tnote{*a}\\\hline
  		$\g_{2(2)}$&$0$&$s(e_1-e_2)s(f_1-f_2)$\\\hline
  		$\e_{6(6)}$&$(3/2,1/2,-1/2,-3/2)$&$s(e_1+e_4)s(e_2+e_3)$\\\hline
  		$\e_{7(7)}$&$(3/2,1/2,-1/2,-3/2,$&$s(e_1-e_5)s(e_2-e_6)s(e_3-e_7)s(e_4-e_8)$\\
  		&$3/2,1/2,-1/2,-3/2)$\\\hline
  		$\e_{8(8)}$&$(7/2,5/2,\cdots,-7/2)$&$s(e_1+e_8)s(e_2+e_7)s(e_3+e_6)s(e_4+e_5)$\\\bottomrule
  	\end{tabular}
	\begin{tablenotes}
	  \item[*] We write $\xi_0$ for the $(\R\beta_{ss})^{\perp}$-component of $\mu_0+\rho_{\mathfrak{k}_{ss}}$. See Lemma $\ref{two-lines}$ for the definition of $w_0$. 
	  \item[*] $s(\eta)$ denotes the reflection with respect to the root $\eta$, and the standard basis with respect to the first (resp. second) component of $K$ in Table $\ref{classification2}$ is denoted by 
	  $e_i$ (resp. $f_i$). 
	  \item[*a] $\eta_1:=\frac{1}{2}(e_1-e_2-e_3+e_4+e_5-e_6+e_7-e_8)$, 
	  $\eta_2:=\frac{1}{2}(-e_1+e_2+e_3-e_4+e_5-e_6+e_7-e_8)$.
	\end{tablenotes}
	\end{threeparttable}
    \end{center}
   \end{table}
\end{landscape}
\end{document}